\documentclass[10pt,leqno,fleqn]{amsart}
\usepackage[english]{babel}
\usepackage{amsmath, amsthm, amsfonts, mathrsfs, amssymb, float}
\usepackage{mathtools}
\mathtoolsset{centercolon}
\usepackage{booktabs}
\usepackage[shortlabels]{enumitem}
\setlist[itemize]{leftmargin=20pt}
\usepackage[colorlinks]{hyperref}
\usepackage{tikz}
\usetikzlibrary{arrows}
\usepackage{pifont}
\usepackage{color}
\usepackage{graphicx}

\mathtoolsset{showonlyrefs,showmanualtags}
\usepackage[msc-links]{amsrefs}

\renewcommand{\emptyset}{\varnothing}

\newtheorem{theorem}{Theorem}

\newtheorem{lemma}[theorem]{Lemma}

\theoremstyle{remark}

\theoremstyle{definition}
\newtheorem{definition}[theorem]{Definition}

\numberwithin{theorem}{section}
\numberwithin{equation}{section}

\title[Characterizations of $A_\infty$ Weights]{Characterizations of $A_\infty$ Weights in Ergodic Theory}

\author[W. Chen]{Wei Chen}
\address{ School of Mathematical Sciences, Yangzhou University, Yangzhou 225002, China}
\email {weichen@yzu.edu.cn}

\author[J. Y. Wang]{Jingyi Wang}
\address{ School of Mathematical Sciences, Yangzhou University, Yangzhou 225002, China}
\email {jywang$\_$yzu@163.com}

\thanks{W. Chen is supported by the National Natural Science Foundation of China(11971419, 12271469).}

\allowdisplaybreaks
\begin{document}
 \begin{abstract}We establish a discrete weighted version of Calder\'{o}n-Zygmund decomposition from the perspective of dyadic grid in ergodic theory. Based on the decomposition, we study discrete $A_\infty$ weights. First, characterizations of
the reverse H\"{o}lder's inequality and their extensions are obtained. 
Second, the properties of $A_\infty$ are given, specifically $A_\infty$ implies the reverse H\"{o}lder's inequality. Finally, 
under a doubling condition on weights, $A_\infty$ follows from the reverse H\"{o}lder's inequality. This means that we obtain equivalent characterizations of $A_{\infty}$. Because $A_{\infty}$ implies the doubling condition, it seems reasonable to assume the condition. 
\end{abstract}




\maketitle
\section{Introduction}Let $\omega$ be a non-negative measurable function on $\mathbb{R}^n$ and let $\mu$ be Lebesgue measure. The function $\omega$ is said to be an $A_p$ weight with $p>1,$ if there exists a constant $C$ for all cubes $Q$ such that
$$
\bigg( \frac{1}{|Q|} \int_Q \omega \, d\mu \bigg) \bigg( \frac{1}{|Q|} \int_Q \omega^{-\frac{1}{p-1}} \, d\mu  \bigg)^{p-1} \leq C.
$$
This kind of weight can be probably traced back to \cite{MR160073}, where its analogue was used to studied the summability of Fourier series. Muckenhoupt \cite{MR293384} observed an open property $A_p=\cup_{1<q<p}A_q$ and characterized the boundedness of the Hardy-Littlewood maximal operator in terms of $A_p.$ In addition, Muckenhoupt \cite{MR350297} defined an $A^M_\infty$ weight as follows: there exist $0<\varepsilon,~\delta<1$ such that if $Q$ is a cube, $E\subseteq Q$ and $|E| < \delta |Q|,$ then $\omega(E) < \varepsilon \omega(Q).$ It was shown that $A^M_{\infty}=\cup_{p>1}A_p.$ 
Around the same time,
Coifman and Fefferman \cite{MR358205} introduced an $A^{CF}_\infty$ weight and proved $A^{CF}_{\infty}=\cup_{p>1}A_p,$ where the $A^{CF}_\infty$ weight $\omega$ is defined as follows: there exist $C,~\delta>0$ such that for all $E\subseteq Q$
$$
\frac{\omega(E)}{\omega(Q)} \leq C \bigg( \frac{|E|}{|Q|} \bigg)^\delta.
$$
They defined the reverse H\"{o}lder's inequality, i.e., there exist $C,~q>1$ for all cubes $Q$ such that
$$			 \bigg( \frac{1}{|Q|} \int_Q \omega^{q} \, d\mu  \bigg)^{q} \leq C\bigg( \frac{1}{|Q|}\int_Q \omega \, d\mu \bigg),
$$ 
which was used to prove $A^{CF}_{\infty}=\cup_{p>1}A_p.$
The type of reverse H\"{o}lder's inequality also appeared in the articles of Muckenhoupt \cite{MR293384} and Gehring \cite{MR402038}. Using the inequality,
Muckenhoupt proved the open property of $A_p,$ and Gehring studied partial differential equations and quasi-conformal mappings. 
Later, a condition $A_{\infty}^{exp}$ defined as a limit case of the $A_p$ weight as $p\uparrow\infty$ was studied almost simultaneously in \cite{MR727244} and \cite[p.405]{MR807149}. They 
showed that $\cup_{p>1}A_p=A^{exp}_{\infty}.$ These characterizations of $\cup_{p>1}A_p$ for cubes on $\mathbb{R}^n$ with Lebesgue measure systematically studied in the book of Grafakos \cite[Theorem 7.3.3]{MR3243734}, which contained several other characterizations. 

Besides the setting of cubes on $\mathbb{R}^n$ with Lebesgue measure, Orobitg and P\'erez \cite{MR1881028} gave a substantial analogue of the equivalent theory of $\cup_{p>1}A_p$ when the underly measure is nondoubling but satisfies the faces (or edges) of the cubes
have measure zero.
Recently, Duoandikoetxea, Mart\'{\i}n-Reyes and Ombrosi \cite{MR3473651} compared and discussed much more characterizations of $\cup_{p>1}A_p$ in the setting of general bases, where they established either the truth or falsity of most of the implications between them.  All the unsolved cases were dealt with by Kosz  \cite{MR4446233}.

Motivated by the above work, 
\cite{MR4745061} studied several characterizations of $A_{\infty}$ weights in the setting of martingales. It is well known that analogues (\cite{MR436313}) of the  $A_p$ theory have been developed in this context, but the open property $A_p=\cup_{1<q<p}A_q$ is false in general, because Bonami and L\'{e}pingle \cite{MR544802} showed that for any $p>1,$ there exists a weight $\omega\in A_p,$ but $\omega\notin A_{p-\varepsilon}$ for all $\varepsilon>0.$ In addition, conditional expectations are Radon-Nikod\'{y}m derivatives with respect to sub$\hbox{-}\sigma\hbox{-}$fields which have no geometric structures, so new ingredients are needed, such as the weight modulo conditional expectations and the conditional expectation of tailed maximal operators. 

Most recently, Kinnunen and Myyryl\"ainen (\cite{MR4720961,2310.00370}) developed the topic in the parabolic case. 
The parabolic Muckenhoupt weights are motivated by one
sided maximal functions and a doubly nonlinear parabolic 
partial differential equation of $p$-Laplace type. Applying an uncentered parabolic 
maximal function with a time lag, they \cite{MR4720961} studied characterizations for the parabolic 
$A_\infty$ classes, where several parabolic Calder\'{o}n-Zygmund decompositions, covering and chaining arguments 
appeared. Subsequently they continued the discussion of parabolic H\"{o}lder's inequalities
and Muckenhoupt weights in \cite{2310.00370}, where the challenging features were related to the parabolic geometry and the time lag. 

In this paper, we study discrete version of $A_\infty$ weights in ergodic theory. The theory of weighted inequalities have been extensively in \cite{MR0675431,MR837798,MR782660,MR713735,MR3771482,MR958045,MR1021892,MR3951077,MR3742545}, but 
little is known in the study of $A_\infty$ weights. This paper is an attempt to create the missing theory.  
First, characterizations of
the reverse H\"{o}lder's inequality and their extensions are obtained in Theorems \ref{sThm_equa_RH} and \ref{Thm_equa_RH}, respectively. 
The latter is the bridge between $\cup_{q>1}RH_{q}$ and $\cup_{p>1}A_{p}.$
Second, the properties of $A_\infty$ are given in Theorems \ref{Thm:equa}, \ref{thm:exp-s} and \ref{thm:imp2}, specifically $A_\infty$ implies the reverse H\"{o}lder's inequality. Finally, 
under a doubling condition on weights, we prove that the reverse H\"{o}lder's inequality implies $A_\infty$ in Theorem $\ref{Thm_RH_AP}.$ This means that we obtain equivalent characterizations of $A_{\infty}.$ Our conclusions are shown in Figure \ref{figure_all}.

At the end of the section, we present new ingredients of this paper: the discrete weighted version of Calder\'{o}n-Zygmund decomposition and the doubling condition. The former is constructed from the perspective of dyadic grid. In view of Lemma \ref{lem-dou}, $A_p$ implies the latter, so it seems reasonable to assume the doubling condition. 

\section{Preliminaries}\label{preli_er}
Let $(X,~\mathcal{F},~\mu)$ be a nonatomic, complete probability space and let $T$ be a point transformation mapping from $X$ onto itself. A set $A$ is called invariant if $T(A)=A.$ Then we say that $T$ is ergodic if the only invariant sets are $X$ and $\emptyset.$ Further, $T$
is measure preserving if $\mu(A) = \mu(T^{-1}A).$ In this paper, our transformations are measurable, invertible, ergodic, and measure preserving.

To say that $g$ is a weight means that $g$ is a measurable function with $g>0$ and $\int_{X}gd\mu>0.$ Without loss of generality, we may assume
$\int_{X}gd\mu=1$ since otherwise we can replace $g$ by $g/{\int_{X}gd\mu}.$ For a measurable function $f,$ the weighted averages relative to $g$ are defined by
$$T_{0,k-1}^gf(x)=\left(\sum\limits_{i=0}^{k-1}g(T^{i} x)\right)^{-1}\left(\sum_{i=0}^{k-1} f(T^{i} x)g(T^{i} x)\right),$$   
where $k\in \mathbb{N}$ and $x\in X.$ 

Let us call a set $I$ of consecutive integers an interval and denote the number of integers in $I$ by $\sharp I.$ The more general weighted averages relative to $g$ are defined by
$$T_{I}^gf(x)=\left(\sum\limits_{i\in I}g(T^{i} x)\right)^{-1}\left(\sum_{i\in I} f(T^{i} x)g(T^{i} x)\right),$$ 
where $I$ is an interval and $x\in X.$ Splitting 
$0,1,$ $\ldots, k-1$ into two disjoint intervals  $I_{l}=\{0,1, \ldots\ldots,  [(k-1) / 2]\}$ and $I_{r}=\{[(k-1) / 2]+1, \ldots\ldots,  k-1\},$ we call $I$ the parent interval of $I_{l}$ and $I_{r}.$ In addition, $I_{l}$ and $I_{r}$ are called the left and right children of $I.$
For convenience, we denote $I$ by $\tilde{I_l}$ or $\tilde{I_r}.$ It is easy to check that $\frac{\sharp I}{\sharp I_{r}}\leq\frac{\sharp I}{\sharp I_{l}}\leq 3$ with $\sharp I\geq2.$ Then we have Lemma \ref{gel_CZ} which is the discrete weighted version of Calder\'{o}n-Zygmund decomposition with a fixed $x\in X.$ 

\begin{lemma}\label{gel_CZ} Let $g,~\omega$ be weights and $k\geq 2$.
 Suppose that $\lambda$  is a real number such that
$$\lambda>T_{0,k-1}^g\omega(x),$$
where  $x$  is a fixed point of ~$X$. Then for the set of integers  $0,1,2, \ldots,k-1$  we can choose a (possibly empty) family of intervals  $I_{1}, \ldots, I_{s}$  such that the following holds:
\begin{enumerate}[(a)]
\item \label{G_C_Z_a} For each  $I_{i}, i=1,2, \ldots, s$
$$\lambda< T_{I_i}^g\omega(x) \leqslant C(T,g,I_i,x)\lambda,$$
where  $C(T,g,I_i,x)=\left(\sum_{j \in I_{i}}g(T^{j} x)\right)^{-1}\left(\sum_{j\in\tilde{I_{i}}}g(T^{j} x)\right)
$.

\item \label{G_C_Z_b} If  $j \notin \bigcup\limits_{i=1}^{s} I_{i},$~$0 \leq j \leq k-1 ,$ then  $\omega\left(T^{j} x\right) \leq \lambda $.
\end{enumerate}
\end{lemma}

\begin{proof}[Proof of Lemma \ref{gel_CZ}]	For $k\geq2,$ we consider
$$\left(\sum_{j \in I_{i}}g(T^{j} x)\right)^{-1}\left(\sum_{j \in I_{i}} \omega(T^{j} x)g(T^{j} x)\right), \quad i=l,~r .$$
If this average is bigger than  $\lambda,$ we select this interval and we have
\begin{eqnarray*}
&~&\left(\sum_{j \in I_{i}}g(T^{j} x)\right)^{-1}\left(\sum_{j \in I_{i}} \omega(T^{j} x)g(T^{j} x)\right) \\
&\leqslant& \left[\left(\sum_{j \in I_{i}}g(T^{j} x)\right)^{-1}\left(\sum_{j \in\tilde{ I_{i}}}g(T^{j} x)\right)\right]\times\\&&\times\left[\left(\sum_{j \in\tilde{ I_{i}}}g(T^{j} x)\right)^{-1}\left(\sum_{j \in\tilde{ I_{i}}}\omega(T^{j} x)g(T^{j} x)\right)\right]\\&\leq& \left[\left(\sum_{j \in I_{i}}g(T^{j} x)\right)^{-1}\left(\sum_{j \in\tilde{ I_{i}}}g(T^{j} x)\right)\right]\lambda.
\end{eqnarray*}
If this average is not bigger than  $\lambda$ , we repeat the process. This process will finish in a finite number of steps. The chosen intervals satisfy $\ref{G_C_Z_a}$ and if an integer $t$ is left out, then obviously
$$\omega(T^{t} x) \leqslant \lambda .$$
\end{proof}

For $g\equiv1,$ $T_{0,k-1}^gf(x)$ reduces to the standard averages $k^{-1} \sum_{i=0}^{k-1} f(T^{i} x)$ (see, e.g., Jones \cite{MR430208}) which are denoted by $T_{0,k-1}f(x).$ Jones \cite[Theorem 2.1]{MR430208} established a Calder\'{o}n-Zygmund decomposition in the variable $x\in X$ which is different from Lemma \ref{C_Z}. Letting $g\equiv1$ in Lemma \ref{gel_CZ}, we have $C(T,g,I_j,x)\leqslant3.$ Thus
we have the following Lemma \ref{C_Z} which appeared in \cite[p. 39]{MR0675431}.

\begin{lemma}\label{C_Z}Let $\omega$ be a weight and $k\geq 2$.
 Suppose that $\lambda$  is a real number such that
$$\lambda> T_{0,k-1}\omega(x),$$
where  $x$  is a fixed point of ~$X$. Then for the set of integers  $0,1,2, \ldots,k-1$  we can choose a (possibly empty) family of disjoint subsets  $I_{1}, \ldots, I_{l}$  each of them made up of consecutive integers and such that the following holds:
\begin{enumerate}[(a)]
\item \label{C_Z_a} For each  $I_{i}, i=1,2, \ldots, s,$
$$\lambda<T_{I_i}\omega(x) \leq 3 \lambda.$$

\item \label{C_Z_b} If  $j \notin \bigcup\limits_{i=1}^{s} I_{i},$~$0 \leq j \leq k-1 ,$ then  $\omega\left(T^{j} x\right) \leq \lambda $.
\end{enumerate}
\end{lemma}

\section{Characterizations of the reverse H\"{o}lder's inequality in ergodic theory}\label{Sec-RH}

We study characterizations of the reverse H\"{o}lder's inequality $\cup_{q>1}RH_{q}$ on the probability space $(X,\mathcal{F},\mu).$ This is Theorem \ref{sThm_equa_RH}, which is one of our main resuts. 
\begin{theorem}\label{sThm_equa_RH}Let $w$ be a weight. Then the following statements are
	equivalent.
\begin{enumerate}[\rm (1)]
				
		\item \label{sThme:con}
		There exist $0<\gamma,\delta<1$ such that for $a.e.$~$x$ and for all positive integers $k$ we have
		\begin{equation}\label{sconn}
    	\frac{1}{k}\sharp\Big\{j:0\leq j \leq k-1;\omega(T^{j}x)\leq  \gamma T_{0,k-1}\omega(x)\Big\}\leq\delta,
		\end{equation}
		which is denoted by $\omega\in A^{avg}_{\infty}.$
		
		\item \label{sThm:lambda}There exist $0 < C,\beta< \infty$ such that for $a.e.$~$x$  and for $\lambda>T_{0,k-1}\omega(x)$
        we have
        \begin{equation}\label{slambda1}
    	\sum_{i \in A(\lambda)} \omega(T^{i} x) \leqslant C \lambda  \sharp A(\beta \lambda),
        \end{equation}
        which is denoted by $\omega\in A^{\lambda}_{\infty},$ where $A(\lambda)=\{i: 0 \leqslant i\leqslant k-1; \omega(T^{i}x)>\lambda\}.$
		
		\item \label{sThm:equa_RH} There exist $C,~q>1$ such that for $a.e.$~$x$ and for all positive integers $k$ we have
		\begin{equation}\label{sRH}		
			\Big(\frac{1}{k} \sum\limits_{i=0}^{k-1}\omega^q(T^{i} x)\Big )^{\frac{1}{q}}
			\leqslant C \frac{1}{k} \sum\limits_{i=0}^{k-1} \omega(T^{i} x),
		\end{equation}
		which is the reverse H\"{o}lder's inequality and denoted by $\omega\in \cup_{q>1}RH_{q}$.
				
		\item\label{sThm:d}There exist $1<C<\infty,0<\varepsilon<1$ such that for $a.e.$~$x$, for all positive integers $k$ and all subsets $A$ of $\{0,1,2,\cdots,k-1\}$ we have
		\begin{equation}\label{sd1}
		\frac{	\sum\limits_{i \in A} \omega(T^{i} x)}{\sum\limits_{i=0}^{k-1} \omega(T^{i} x)} \leq C(\frac{\sharp A}{k})^{\varepsilon},
		\end{equation}
	    which is denoted by $\omega\in A^{CF}_{\infty}.$
		
		\item\label{sThm:e}There exist $0<\alpha',\beta'<1$ such that for $a.e.$~$x$, for all positive integers $k$ and all subsets $A$ of $\{0,1,2,\cdots,k-1\}$ we have
		\begin{equation}\label{se1}
		\sum\limits_{i \in A} \omega(T^{i} x) \leq \alpha'\sum\limits_{i=0}^{k-1} \omega(T^{i} x) \Rightarrow \sharp A \leq \beta'k,
		\end{equation}
		which is denoted by $\omega\in \hat{A}^{M}_{\infty}.$

		\item\label{sThm:b}There exist $0<\alpha,\beta<1$ such that for $a.e.$~$x$, for all positive integers $k$ and all subsets $A$ of $\{0,1,2,\cdots,k-1\}$ we have
		\begin{equation}\label{sb1}
			\sharp A \leq \alpha k \Rightarrow	\sum_{i \in A} \omega(T^{i} x) \leq \beta  \sum\limits_{i=0}^{k-1} \omega(T^{i} x),
		\end{equation}
		which is denoted by $\omega\in A^{M}_{\infty}.$		

\item \label{sthm:log}There exists $C>1$ such that for $a.e.$~$x$ and all $k\in \mathbb{N}$ we have
		\begin{equation}\label{sthm:equ_log}
	 \frac{1}{k}\sum\limits_{j=0}^{k-1}\frac{\omega(T^{j}x) }{T_{0, k-1} \omega(x) }\log^+  \frac{\omega(T^{j}x)}{T_{0, k-1} \omega(x)}  \leq C,
		\end{equation}
		which is denoted by $\omega\in A_{\infty}^{log}.$	
	\end{enumerate}
\end{theorem}

The above classes $A^{avg}_{\infty},$ $A^{\lambda}_{\infty}$ and $A_{\infty}^{log}$ in the context of $\mathbb{R}^n$ appeared in \cite{MR358205}, \cite{MR402038} and \cite{MR481968}, respectively.
  We can directly prove Theorem \ref{sThm_equa_RH}, but we develop and prove its analogue, i.e., Theorem \ref{Thm_equa_RH}. The reason is that the analogue is the bridge between $\cup_{q>1}RH_{q}$ and $\cup_{p>1}A_{p}.$ 

\begin{definition}Let $g$ be a weight. We say that
$g$ satisfies the doubling condition
if there exists a constant $C$ such that for all positive integers $k\geq2$ and $x\in X$,  we have
$$\left(\sum_{j \in I_{i}}g(T^{j} x)\right)^{-1}\left(\sum_{j=0}^{k-1}g(T^{j} x)\right)\leq C,$$
where $I_i$ are children of $\{0,1,\ldots,k-1\}.$
\end{definition}

Let $d\hat{\mu}=gd\mu.$ We study the analogue of Theorem \ref{sThm_equa_RH} relative to $(X,\mathcal{F},\hat{\mu})$ which is Theorem \ref{Thm_equa_RH}.

\begin{theorem}\label{Thm_equa_RH} Let $g$ and $\omega$ be weights. If $g$ satisfies the doubling condition, then the following are
	equivalent.
	
	\begin{enumerate}[\rm (1)]
				
		\item \label{Thme:con} 
		There exist $0<\gamma,\delta<1$ such that for $a.e.$~$x$ and for all positive integers $k$ we have
		\begin{equation}\label{conn}
    \left(\sum_{i=0}^{k-1}g(T^{i} x)\right)^{-1}\left(\sum_{j \in \Gamma_{0,k-1}(\gamma)}g(T^{j} x)\right)\leq\delta,
		\end{equation}
		which is denoted by $\omega\in A^{avg}_{\infty}(g),$ where $$\Gamma_{0,k-1}(\gamma)=\left\{j:0\leq j \leq k-1;\omega(T^{j}x)\leq  \gamma T_{0,k-1}^g\omega(x)\right\}.$$
		
		\item \label{Thm:lambda}There exist $0 < C,\beta< \infty$ such that for $a.e.$~$x$  and for $\lambda>T^g_{0,k-1}\omega(x)$
        we have
        \begin{equation}\label{lambda1}
    	\sum_{i \in A(\lambda)} \omega(T^{i} x)g(T^{i} x) \leq C \lambda  \sum_{i \in A(\beta\lambda)} g(T^{i} x),
        \end{equation}
        which is denoted by $\omega\in A^{\lambda}_{\infty}(g),$ where $A(\lambda)=\{i: 0 \leq i\leq k-1; \omega(T^{i}x)>\lambda\}.$
		
		\item \label{Thm:equa_RH} There exist $C,~q>1$ such that for $a.e.$~$x$ and for all positive integers $k$ we have
		\begin{equation}\label{RH}		
			\Big(T_{0,k-1}^g\omega^q(x)\Big )^{\frac{1}{q}}
			\leqslant C T_{0,k-1}^g\omega(x),
		\end{equation}
		which is the reverse H\"{o}lder's inequality and denoted by $\omega\in \cup_{q>1}RH_{q}(g)$.

		\item\label{Thm:d}There exist $1<C<\infty,0<\varepsilon<1$ such that for $a.e.$~$x$, for all positive integers $k$ and all subsets $A$ of $\{0,1,2,\cdots,k-1\}$ we have
		\begin{equation}\label{d1}
		\frac{\sum\limits_{i \in A} \omega(T^{i} x)g(T^{i} x)}{\sum\limits_{i=0}^{k-1} \omega(T^{i} x)g(T^{i} x)} \leq C\left(\frac{\sum\limits_{i \in A} g(T^{i} x)}{\sum\limits_{i=0}^{k-1} g(T^{i} x)}\right)^{\varepsilon},
		\end{equation}
	    which is denoted by $\omega\in A^{CF}_{\infty}(g).$
		
		\item\label{Thm:e}There exist $0<\alpha',\beta'<1$ such that for $a.e.$~$x$, for all positive integers $k$ and all subsets $A$ of $\{0,1,2,\cdots,k-1\}$ we have
		\begin{equation}\label{e1}
		\sum\limits_{i \in A} \omega(T^{i} x)g(T^{i} x) \leq \alpha'\sum\limits_{i=0}^{k-1} \omega(T^{i} x)g(T^{i} x) \Rightarrow \sum\limits_{i \in A} g(T^{i} x) \leq \beta'\sum\limits_{i=0}^{k-1} g(T^{i} x),
		\end{equation}
		which is denoted by $\omega\in \hat{A}^{M}_{\infty}(g).$

		\item\label{Thm:b}There exist $0<\alpha,\beta<1$ such that for $a.e.$~$x$, for all positive integers $k$ and all subsets $A$ of $\{0,1,2,\cdots,k-1\}$ we have
		\begin{equation}\label{b1}
			\sum\limits_{i \in A} g(T^{i} x) \leq \alpha \sum\limits_{i=0}^{k-1} g(T^{i} x) \Rightarrow \sum\limits_{i \in A} \omega(T^{i} x)g(T^{i} x) \leq \beta  \sum\limits_{i=0}^{k-1} \omega(T^{i} x)g(T^{i} x),
		\end{equation}
		which is denoted by $\omega\in A^{M}_{\infty}(g).$	

		\item \label{thm:log}There exists $C>1$ such that for $a.e.$~$x$ and $k\in \mathbb{N}$ we have
		\begin{equation}\label{thm:equ_log}
	 \frac{1}{\sum_{i=0}^{k-1}g(T^{i} x)}\sum\limits_{j=0}^{k-1}\left(\frac{\omega(T^{j}x) }{T^g_{0, k-1} \omega(x) }\log^+  \frac{\omega(T^{j}x)}{T^g_{0, k-1} \omega(x)}\right)g(T^{j} x)  \leq C,
		\end{equation}
		which is denoted by $\omega\in A_{\infty}^{log}(g).$	
	\end{enumerate}
\end{theorem}

\begin{proof}[Proof of Theorem \ref{Thm_equa_RH}] We shall follow the schemes:
$$\ref{Thme:con}\Rightarrow\ref{Thm:lambda}\Rightarrow\ref{Thm:equa_RH}\Rightarrow\ref{Thm:d}\Rightarrow\ref{Thm:e}\Rightarrow\ref{Thm:b}\Rightarrow\ref{Thme:con}$$
and
$$\ref{Thm:equa_RH}\Rightarrow\ref{thm:log}\Rightarrow\ref{Thm:b}.$$

$\ref{Thme:con}\Rightarrow\ref{Thm:lambda}$ Let $\beta=\gamma$ and $\alpha=1-\delta.$ We obtain that \ref{Thme:con} is equivalent to
		\begin{equation*}
\left(\sum_{i=0}^{k-1}g(T^{i} x)\right)^{-1}\left(\sum_{j \in \Gamma^c_{0,k-1}(\beta)}g(T^{j} x)\right)>\alpha,
		\end{equation*}
where $$\Gamma^c_{0,k-1}(\beta)=\left\{j:0\leq j \leq k-1;\omega(T^{j}x)>  \beta T_{0,k-1}^g\omega(x)\right\}.$$
Then 
\begin{equation*}
\left(\sum_{i\in I}g(T^{i} x)\right)^{-1}\left(\sum_{j \in \Gamma^c_{I}(\beta)}g(T^{j} x)\right)>\alpha,
		\end{equation*}
where $\Gamma^c_{I}(\beta)=\left\{j:j\in I;\omega(T^{j}x)>  \beta T_{I}^g\omega(x)\right\}$
and $I$ is any other interval instead of $0\leq i \leq k-1.$

 Let  $\lambda$  be a positive number such that
$$\lambda>T^g_{0, k-1} \omega(x).$$
We estimate $ \sum\limits_{i\in A(\lambda)} \omega(T^{i} x)g(T^{i} x).$ Using the discrete weighted version of Calder\'{o}n-Zygmund decomposition Lemma \ref{gel_CZ} for this  $\lambda$, we have a family of disjoint intervals  $I_{j}$ satisfying \ref{G_C_Z_a} and \ref{G_C_Z_b} of the said decomposition, so
\begin{eqnarray*}
A(\lambda) =\{i: 0 \leq  i\leq  k-1; \omega(T^{i}x)>\lambda\}\subset \cup_{j} I_{j},
\end{eqnarray*}
where we have used \ref{G_C_Z_b}. 
Using the doubling condition, we obtain that
\begin{eqnarray*}
	\sum_{i \in A(\lambda)} \omega(T^{i}x)g(T^{i}x) &\leq & \sum_{j} \sum_{i \in I_{j}} \omega(T^{i}x)g(T^{i}x) \\
     &\leq & C\sum_{j}  \lambda\left(\sum_{i\in  I_{j}}g(T^{i} x)\right)\\
	&\leq & C \lambda \sum_{j} \alpha^{-1} \left(\sum_{i \in \Gamma^c_{I_j}(\beta)}g(T^{i} x)\right)\\
	&=& C\alpha^{-1} \lambda  \sum_{j} \left(\sum_{i \in \Gamma^c_{I_j}(\beta)}g(T^{i} x)\right).
\end{eqnarray*}
It follows from \ref{G_C_Z_a} that 
\begin{eqnarray*}
\Gamma^c_{I_j}(\beta)&=&\left\{i:i\in I_j;\omega(T^{i}x)>  \beta T_{I}^g\omega(x)\right\}\\
&\subset&\left\{i:i\in I_j;\omega(T^{i}x)>  \beta \lambda\right\}.\end{eqnarray*}
Thus 
\begin{eqnarray*}
	\sum_{i \in A(\lambda)} \omega(T^{i}x)g(T^{i}x) 
 &\leq&C\alpha^{-1} \lambda  \sum_{j} \left(\sum_{i \in \Gamma^c_{I_j}(\beta)}g(T^{i} x)\right)\\
    	&\leq& C\alpha^{-1} \lambda  \sum_{i \in A(\beta\lambda)} g(T^{i} x) .
\end{eqnarray*}
Then we have $\eqref{lambda1}.$ 

$\ref{Thm:lambda}\Rightarrow\ref{Thm:equa_RH}$
	For all positive integers $k$ and $\lambda>0$, let $\Lambda(\lambda)=\{i: 0 \leq  i\leq  k-1;~\frac{\omega(T^{i}x)}{T^g_{0, k-1} \omega(x)}>\lambda\}.$ Then
	\begin{eqnarray*}
		&&\sum\limits_{i=0}^{k-1}\left(\frac{\omega(T^{i}x)}{T^g_{0, k-1} \omega(x)}\right)^{1+\delta}g(T^{i}x)\\
		&=&\sum\limits_{i=0}^{k-1}\left(\frac{\omega(T^{i}x)}{T^g_{0, k-1} \omega(x)}\right)^{\delta}\frac{\omega(T^{i}x)}{T^g_{0, k-1} \omega(x)}g(T^{i}x)\\
		&=&\delta\int_0^{+\infty}\lambda^{\delta-1}
		\sum\limits_{i \in \Lambda(\lambda)} \frac{\omega(T^{i}x)g(T^{i}x)}{T^g_{0, k-1} \omega(x)}d\lambda\\
		&=&\delta\int_0^{1}\lambda^{\delta-1}
		\sum\limits_{i \in \Lambda(\lambda)}\frac{ \omega(T^{i}x)g(T^{i}x)}{T^g_{0, k-1} \omega(x)}d\lambda+\delta\int_1^{+\infty}\lambda^{\delta-1}
		\sum\limits_{i \in \Lambda(\lambda)}\frac{ \omega(T^{i}x)g(T^{i}x)}{T^g_{0, k-1} \omega(x)}d\lambda.
	\end{eqnarray*}
	It follows that
	\begin{eqnarray*}
		\delta\int_0^{1}\lambda^{\delta-1}
		\frac{\sum\limits_{i \in \Lambda(\lambda)} \omega(T^{i}x)g(T^{i}x)}{T^g_{0, k-1} \omega(x)}d\lambda 
		\leq\frac{\sum\limits_{i=0}^{k-1} \omega(T^{i}x)g(T^{i}x)}{T^g_{0, k-1} \omega(x)}\delta\int_0^{1}\lambda^{\delta-1}
		d\lambda=\sum\limits_{i=0}^{k-1}g(T^{j}x).
	\end{eqnarray*}

For $\lambda>1,$ let $\tilde{\lambda}=\lambda T^g_{0, k-1} \omega(x).$ Then $A(\tilde{\lambda})=\Lambda(\lambda).$
Using \eqref{lambda1}, we obtain the following estimate
	\begin{eqnarray*}
		&~&\delta\int_1^{+\infty}\lambda^{\delta-1}
		\sum\limits_{i \in \Lambda(\lambda)}\frac{ \omega(T^{i}x)g(T^{i}x)}{T^g_{0, k-1} \omega(x)}d\lambda\\
           &=&\delta\int_1^{+\infty}\lambda^{\delta-1}
		\frac{\sum\limits_{i \in A(\tilde{\lambda})} \omega(T^{i}x)g(T^{i}x)}{T^g_{0, k-1} \omega(x)}d\lambda\\
		&\leq&C\delta\int_1^{+\infty}\lambda^{\delta-1} \tilde{\lambda}\frac{\sum_{i \in A(\beta\tilde{\lambda})} g(T^{i} x)}{T^g_{0, k-1} \omega(x)}d\lambda \\
          &=&C\delta\int_1^{+\infty}\lambda^{\delta}\sum_{i \in A(\beta\tilde{\lambda})} g(T^{i} x)
		d\lambda \\
		&=&\frac{C\delta}{\beta^{1+\delta}}\int_\beta^{+\infty}\lambda^{\delta}\sum_{i \in A(\tilde{\lambda})} g(T^{i} x)
		d\lambda\\
          &\leq&\frac{C\delta}{\beta^{1+\delta}}\int_0^{+\infty}\lambda^{\delta}\sum_{i \in \Lambda(\lambda)} g(T^{i} x)
		d\lambda\\
	    &=&\frac{C\delta}{(1+\delta)\beta^{1+\delta}}\sum\limits_{j=0}^{k-1}\left(\frac{\omega(T^{j}x)}{T^g_{0, k-1} \omega(x)}\right)^{1+\delta}g(T^{j} x).
	\end{eqnarray*}
	Because of $\lim\limits_{\delta\rightarrow0}\frac{C\delta}{(1+\delta)\beta^{1+\delta}}=0,$
	we can choose $\delta\in(0,~1)$ such that $\frac{C\delta}{(1+\delta)\beta^{1+\delta}}<\frac{1}{2}.$ Then we have
	$$\sum\limits_{j=0}^{k-1}\Big(\frac{\omega(T^{j}x)}{T^g_{0, k-1} \omega(x)}\Big)^{1+\delta}g(T^{j}x)\leq2\sum\limits_{j=0}^{k-1}g(T^{j}x).$$
	It follows that $$\left[\left(\sum\limits_{i=0}^{k-1}g(T^{i}x)\right)^{-1}\left(\sum\limits_{j=0}^{k-1}\omega^{1+\delta}(T^{j}x)g(T^{j}x)\right)\right]^{\frac{1}{1+\delta}}\leq2^{\frac{1}{1+\delta}}T^g_{0, k-1} \omega(x),$$
which  is exactly \eqref{RH}	 with $q=1+\delta$ and $C=2^{\frac{1}{1+\delta}}.$

$\ref{Thm:equa_RH}\Rightarrow\ref{Thm:d}$ Let $\omega\in RH_q$ with $q>1.$ Denoting $\varepsilon_1=q-1,$ we apply H\"{o}lder's inequality
	\begin{eqnarray*}
&~&\sum\limits_{i\in A}\omega(T^{i}x)g(T^{i}x)\\
		&\leq& \left(\sum\limits_{i\in A}\omega^{1+\varepsilon_1}(T^{i}x)g(T^{i}x)\right)^{\frac{1}{1+\varepsilon_1}}\left(\sum_{i\in A}g(T^{i}x)\right)^{\frac{\varepsilon_1}{1+\varepsilon_1}}\\		
		&\leq& \left[\left(\sum\limits_{i=0}^{k-1}g(T^{i}x)\right)^{-1}\left(\sum\limits_{i=0}^{k-1} \omega^{1+\varepsilon_1}(T^{i}x)g(T^ix)\right)\right]^{\frac{1}{1+\varepsilon_1}}\times\\
&&\times\left(\sum\limits_{i=0}^{k-1}g(T^{i}x)\right)^{\frac{1}{1+\varepsilon_1}}
\left(\sum_{i\in A}g(T^{i}x)\right)^{\frac{\varepsilon_1}{1+\varepsilon_1}}\\
&\leq&C\left(\sum\limits_{i=0}^{k-1}g(T^{i}x)\right)^{-1}\left(\sum\limits_{i=0}^{k-1} \omega(T^{i}x)g(T^ix)\right)\times\\
&&\times\left(\sum\limits_{i=0}^{k-1}g(T^{i}x)\right)^{\frac{1}{1+\varepsilon_1}}
\left(\sum_{i\in A}g(T^{i}x)\right)^{\frac{\varepsilon_1}{1+\varepsilon_1}}\\
&=&C\left(\sum\limits_{i=0}^{k-1} \omega(T^{i}x)g(T^ix)\right)\left[\left(\sum\limits_{i=0}^{k-1}g(T^{i}x)\right)^{-1}
\left(\sum_{i\in A}g(T^{i}x)\right)\right]^{\frac{\varepsilon_1}{1+\varepsilon_1}},
	\end{eqnarray*}	
where we have used the reverse H\"{o}lder's inequality. Thus we have
\begin{eqnarray*}
&&\sum\limits_{i\in A}\omega(T^{i}x)g(T^{i}x)\\
&\leq&C\left(\sum\limits_{i=0}^{k-1} \omega(T^{i}x)g(T^ix)\right)\left[\left(\sum\limits_{i=0}^{k-1}g(T^{i}x)\right)^{-1}
\left(\sum_{i\in A}g(T^{i}x)\right)\right]^{\frac{\varepsilon_1}{1+\varepsilon_1}},
\end{eqnarray*}
which implies $\eqref{d1}$ with  $\varepsilon=\frac{\varepsilon_1}{1+\varepsilon_1}.$

$\ref{Thm:d}\Rightarrow\ref{Thm:e}$
	Pick an $0<\alpha^{''}<1$ such that $\beta^{''}=C(\alpha^{''})^{\varepsilon}<1$.
	For $A\subseteq \{0,1,2,\dots,k-1\}$, denote $A^{c}=\{0,1,2,\dots,k-1\}\backslash A$.
	It follows from $\ref{Thm:d}$ that
	\begin{equation}
		\sum\limits_{i \in A^c} g(T^{i} x)\leq\alpha^{''}\sum\limits_{i=0}^{k-1} g(T^{i} x) \Rightarrow
		\sum\limits_{i \in A^c} \omega(T^{i} x)g(T^{i} x)\leq\beta^{''}\sum\limits_{i=0}^{k-1} \omega(T^{i}x)g(T^{i}x),
	\end{equation} which can be equivalently written as
	\begin{eqnarray*}
&&\sum\limits_{i \in A^c} g(T^{i} x)>(1-\alpha^{''})\sum\limits_{i=0}^{k-1} g(T^{i} x) \Rightarrow\\
&\Rightarrow&
		\sum\limits_{i \in A^c} \omega(T^{i} x)g(T^{i} x)>(1-\beta^{''})\sum\limits_{i=0}^{k-1} \omega(T^{i}x)g(T^{i}x).
	\end{eqnarray*}
	In other words, for subsets A of $\{0,1,2,\dots,k-1\}$ we have
	\begin{eqnarray*}
		&&\sum\limits_{i \in A} \omega(T^{i} x)g(T^{i} x)\leq(1-\beta^{''})\sum\limits_{i=0}^{k-1} \omega(T^{i}x)g(T^{i}x)\Rightarrow\\
&\Rightarrow&\sum\limits_{i \in A} g(T^{i} x)\leq(1-\alpha^{''})\sum\limits_{i=0}^{k-1} g(T^{i} x).
	\end{eqnarray*}
	which is the statement in $\eqref{e1}$ if we set $\alpha^{'}=1-\beta^{''}$ and $\beta^{'}=1-\alpha^{''}.$

$\ref{Thm:e}\Rightarrow\ref{Thm:b}$ Let $E$ be a subsets of $\{1,2,\cdots,k-1\}$ such that  $$\sum\limits_{i\in E} \omega(T^{i}x)g(T^{i}x)> \beta \sum\limits_{i=0}^{k-1}\omega(T^{i}x)g(T^{i}x), $$ where  $\beta$  will be chosen later. Set $ S= \{1,2,\cdots,k-1\}\backslash E $. Then $$\sum\limits_{i\in S} \omega(T^{i}x)g(T^{i}x)\leq(1-\beta) \sum\limits_{i=0}^{k-1} \omega(T^{i}x)g(T^{i}x).$$ Writing
\begin{equation}
		S_{1}=\left\{i \in S: \omega(T^{i}x)>\alpha'T^g_{0, k-1} \omega(x)\right\},  S_{2}=S \backslash S_{1} ,	
\end{equation}
we will estimate $\sum\limits_{i \in S_1} g(T^{i} x)$ and $\sum\limits_{i \in S_2} g(T^{i} x).$
For $S_1,$ it is clear that 
\begin{equation}
		\sum\limits_{i\in S_1}g(T^ix)  < \frac{1}{\alpha'T^g_{0, k-1} \omega(x)} \sum\limits_{i\in S}\omega(T^{i}x)g(T^ix) \leq \frac{1-\beta}{\alpha'} \sum\limits_{i=0}^{k-1} g(T^{i} x) .
\end{equation}	
For $S_2,$ we obtain that
\begin{eqnarray*} \sum\limits_{i \in S_2} \omega(T^{i} x)g(T^{i} x)&\leq& \alpha'T^g_{0, k-1} \omega(x)\sum\limits_{i \in S_2}g(T^{i} x)\\
&\leq& \alpha'T^g_{0, k-1} \omega(x)\sum\limits_{i=0}^{k-1}g(T^{i} x)\\
&=&\alpha' \sum\limits_{i=0}^{k-1} \omega(T^{i} x)g(T^{i} x). \end{eqnarray*}
Furthermore, $ \sum\limits_{i\in S_{2}} g(T^{i}x)\leq \beta' \sum\limits_{i=0}^{k-1} g(T^{i} x) $ by the assumption. 
Because of $\lim\limits_{\beta\rightarrow 1-}\beta'+\frac{1-\beta}{\alpha'}=\beta'<1,$
we choose $\alpha , \beta \in (0,1) $ such that  $\beta'+\frac{1-\beta}{\alpha'}\leq1-\alpha.$ Thus $$\sum\limits_{i\in S}g(T^ix) \leq(1-\alpha) \sum\limits_{i=0}^{k-1} g(T^{i} x) ,$$ which implies $\sum\limits_{i\in E}g(T^ix)>\alpha \sum\limits_{i=0}^{k-1} g(T^{i} x) .$

\ref{Thm:b} $\Rightarrow$ \ref{Thme:con} 
Let $0<\gamma \leq1-\beta.$ Setting 
$$E=\{j:0\leq j\leq k-1;\omega(T^{j}x)\leq\gamma T^g_{0, k-1}\omega(x)\},$$ we have
	\begin{eqnarray*}
\sum\limits_{j \in E} \omega(T^{j}x)g(T^{j}x)&\leq&\gamma T^g_{0, k-1}\omega(x)\sum\limits_{i\in E}g(T^ix)\\
&\leq&\gamma T^g_{0, k-1}\omega(x)\sum\limits_{i=0}^{k-1}g(T^ix)\\
&=&\gamma\sum\limits_{i=0}^{k-1} \omega(T^{i}x)g(T^{i}x)\\
&\leq&(1-\beta)\sum\limits_{i=0}^{k-1} \omega(T^{i}x)g(T^{i}x).
	\end{eqnarray*}
For $E,$ denote $E^{c}=\{0,1,2,\dots,k-1\}\backslash E.$
It follows that $\frac{\sum\limits_{j \in E^c} \omega(T^{j}x)g(T^{j}x)}{\sum\limits_{i=0}^{k-1} \omega(T^{i}x)g(T^{i}x)}>\beta.$ In view of \ref{Thm:b}, we obtain that $$\frac{\sum\limits_{j \in E^c} g(T^{j}x)}{\sum\limits_{i=0}^{k-1}g(T^{i}x)}>\alpha,$$
which implies $\frac{\sum\limits_{j \in E} g(T^{j}x)}{\sum\limits_{i=0}^{k-1}g(T^{i}x)}\leq1-\alpha.$
Thus \ref{Thme:con} is valid with $1-\alpha=\delta.$

	$\ref{Thm:equa_RH} \Rightarrow\ref{thm:log}$
	Let $E_l=\{j: 0 \leq j\leq k-1;2^l<\frac{\omega(T^{j}x)}{T^g_{0, k-1} \omega(x)}\leq2^{l+1}\}$ for $l\in \mathbb{N}.$ In view of \ref{Thm:equa_RH}, we have
	\begin{eqnarray*}
		2^{lp}\frac{\sum\limits_{j \in E_l}g(T^jx)}{\sum\limits_{i=0}^{k-1}g(T^{i}x)}
		&\leq&\frac{1}{\sum\limits_{i=0}^{k-1}g(T^{i}x)}\sum\limits_{j\in E_l}\left(\frac{\omega(T^{j}x)}{T^g_{0, k-1} \omega(x)}\right)^{p}g(T^{j}x)\\
&\leq&\frac{1}{\sum\limits_{i=0}^{k-1}g(T^{i}x)}\sum\limits_{j=0}^{k-1}\left(\frac{\omega(T^{j}x)}{T^g_{0, k-1} \omega(x)}\right)^{p}g(T^{j}x)\\&\leq& C^p,
	\end{eqnarray*}
	which implies $\frac{\sum\limits_{j \in E_l}g(T^jx)}{\sum\limits_{i=0}^{k-1}g(T^{i}x)}\leq C^p2^{-lp}.$
	It follows that
	\begin{eqnarray*}
		&~&\frac{1}{\sum\limits_{i=0}^{k-1}g(T^{i}x)}\sum\limits_{j=0}^{k-1}\left( \frac{\omega(T^{j}x) }{T^g_{0, k-1} \omega(x) }\log^+  \frac{\omega(T^{j}x)}{T^g_{0, k-1} \omega(x)}\right)g(T^{j}x)\\
		&=&\frac{1}{\sum\limits_{i=0}^{k-1}g(T^{i}x)}\sum\limits_{l=0}^{+\infty}\sum\limits_{j\in E_l } \left(\frac{\omega(T^{j}x) }{T^g_{0, k-1} \omega(x) }\log^+  \frac{\omega(T^{j}x)}{T^g_{0, k-1} \omega(x)}\right)g(T^{j}x)\\
		&\leq&\sum\limits_{l=0}^{+\infty}2^{l+1} (l+1)\frac{\sum\limits_{j \in E_l}g(T^jx)}{\sum\limits_{i=0}^{k-1}g(T^{i}x)}\\
		&\leq&C^p\sum\limits_{l=0}^{+\infty}(l+1)2^{l+1}2^{-lp},
	\end{eqnarray*}
where the series $\sum\limits_{l=0}^{+\infty}(l+1)2^{l+1}2^{-lq}$ is convergent. Then we have \eqref{thm:equ_log}.

	$\ref{thm:log}\Rightarrow\ref{Thm:b}$
	Let $	\frac{\sum\limits_{j \in A}g(T^{j}x)}{\sum\limits_{i=0}^{k-1} g(T^{i} x)} \leq\alpha<1.$ Recall that $ab\leq a\log a-a+e^b$ where $a>1$ and $b\geq0.$ Let $B=\{j: 0 \leq j\leq k-1:\frac{\omega(T^{j}x)}{T^g_{0, k-1} \omega(x)}\leq1\}$ and $B^c=\{j: 0 \leq j\leq k-1:\frac{\omega(T^{j}x)}{T^g_{0, k-1} \omega(x)}>1\}.$
	Then
	\begin{eqnarray*}
	  &~&\frac{\sum\limits_{j \in A} \omega(T^{j}x)g(T^{j}x)}{\sum\limits_{i=0}^{k-1} \omega(T^{i} x)g(T^{i}x)}\\
		 &= &\frac{1}{\sum\limits_{i=0}^{k-1} g(T^{i} x)}\sum\limits_{j \in A\cap B} \frac{\omega(T^{j}x)}{T^g_{0, k-1}\omega(x)}g(T^{j} x)+\frac{1}{\sum\limits_{i=0}^{k-1}g(T^{i}x)}\sum\limits_{j \in A\cap B^c }  \frac{\omega(T^{j}x)}{T^g_{0, k-1} \omega(x)} g(T^{j} x)\\
		&\leq&
		\frac{\sum\limits_{j \in A} g(T^jx)}{\sum\limits_{i=0}^{k-1} g(T^{i} x)}+\\
&&+\frac{1}{\sum\limits_{i=0}^{k-1} g(T^{i} x)}
\frac{1}{b+1}\sum\limits_{j \in A\cap B^c} \left(  \frac{\omega(T^{j}x)}{T^g_{0, k-1} \omega(x)}\log \frac{\omega(T^{j}x)}{T^g_{0, k-1} \omega(x)}+e^b \right)g(T^jx)\\
		&\leq&\alpha+\frac{1}{b+1}\frac{1}{\sum\limits_{i=0}^{k-1} g(T^{i} x)}\sum\limits_{i=0}^{k-1}\left(\frac{\omega(T^{i}x)}{T^g_{0, k-1} \omega(x)}\log^+\frac{\omega(T^{i}x)}{T^g_{0, k-1} \omega(x)}\right)g(T^ix)+ \\
&&+\frac{e^b}{b+1}\frac{\sum\limits_{j \in A} g(T^jx)}{\sum\limits_{i=0}^{k-1} g(T^{i} x)} \\
		&\leq&\alpha(1+\frac{e^b}{b+1})+\frac{C}{b+1}.
	\end{eqnarray*}
	Setting $b=2C-1,$ we can pick an $\alpha$ small enough that
	$\alpha(1+\frac{e^b}{b+1})\leq\frac{1}{4}$ because of $\lim\limits_{\alpha\rightarrow0}\alpha(1+\frac{e^b}{b+1})=0.$
	Thus $\frac{\sum\limits_{i \in A} \omega(T^{i}x)}{\sum\limits_{i=0}^{k-1} \omega(T^{i}x)}\leq\frac{3}{4}.$
\end{proof}

Letting $g\equiv1$ in Theorem \ref{Thm_equa_RH}, we have Theorem \ref{sThm_equa_RH}, which is shown in Figure \ref{figure_RH}.

\begin{center}
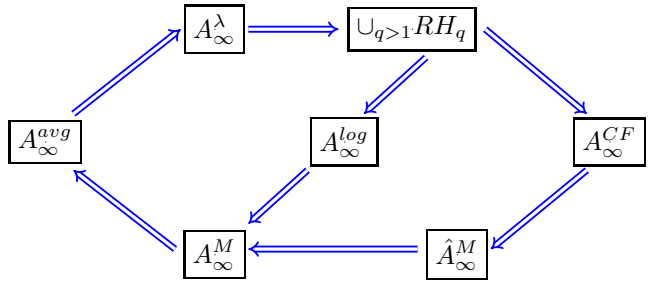
\begin{figure}[H]
	\begin{tikzpicture}[scale=0.75]
	\draw (1,3) node {\boxed{A^{avg}_{\infty}}} -- (1,3);
	\draw (4,5) node {\boxed{A^{\lambda}_{\infty}}} -- (4,5);
	\draw (7.5,5) node {\boxed{\cup_{q>1}RH_{q}}} -- (7.5,5);
    \draw (11,3) node  {\boxed{A^{CF}_{\infty}}}-- (11,3);
    \draw (8.3,1) node  {\boxed{ \hat{A}^{M}_{\infty}}}-- (8.3,1);
	\draw (4,1) node  {\boxed{ A^{M}_{\infty}}}-- (4,1);
	\draw (6.3,3) node {\boxed{A^{log}_{\infty}}} -- (6.3,3);
    \draw[blue, -implies,double equal sign distance, line width=0.25mm] (1.5,3.5) -- (3.4,5);
    \draw[blue, -implies,double equal sign distance, line width=0.25mm] (4.6,5) -- (6.2,5);
    \draw[blue, -implies,double equal sign distance, line width=0.25mm] (8.8,5) -- (10.6,3.5);
    \draw[blue, -implies,double equal sign distance, line width=0.25mm] (10.6,2.5) -- (8.9,1.1);
    \draw[blue, -implies,double equal sign distance, line width=0.25mm] (7.6,1.1) -- (4.6,1.1);
    \draw[blue, -implies,double equal sign distance, line width=0.25mm] (3.3,1.1) -- (1.5,2.5);    
\draw[blue, -implies,double equal sign distance, line width=0.25mm] (7.75,4.5) -- (6.65,3.5);
\draw[blue, -implies,double equal sign distance, line width=0.25mm] (5.65,2.5) -- (4.6,1.55);
	\end{tikzpicture}
\caption{Characterizations of the Reverse H\"{o}lder's Inequality}
	\label{figure_RH}
	\end{figure}
\end{center}

\section{Properties of the union of $A_p$ weights in ergodic theory}
In this section, we prove Theorems \ref{Thm:equa}, \ref{thm:exp-s} and \ref{thm:imp2}.
These are properties of the union of $A_p$ weights on $(X,\mathcal{F},\mu).$ 
\begin{theorem}\label{Thm:equa}Let $w$ be a weight.  We have the sequence of implications
	$\ref{Thm:equa_Ap}\Rightarrow\ref{Thm:equa_A_exp_infty} \Rightarrow\ref{Thm:con}.$
	\begin{enumerate}[\rm (1)]
		
        \item \label{Thm:equa_Ap}There exist
		$C,~p>1$ such that for $a.e.$~$x$ and for all positive integers $k$ we have
		\begin{equation}\label{Ap}
	    \frac{1}{k}\sum_{i=0}^{k-1} \omega(T^{i} x)\left(\frac{1}{k} \sum_{i=0}^{k-1}\omega^{-\frac{1}{p-1}}(T^{i} x)\right)^{p-1} \leq C,
		\end{equation}
		which is denoted by $\omega\in\cup_{p>1}A_p$ or $A_\infty.$ 
				
		\item \label{Thm:equa_A_exp_infty}There exists a positive
		constant $C$ such that for $a.e.$~$x$ and for all positive integers $k$ we have
		\begin{equation}\label{A_exp_infty}
		\frac{1}{k}\sum\limits_{i=0}^{k-1} \omega(T^{i}x)\exp\left({\frac{1}{k}\sum\limits_{i=0}^{k-1}\log \omega^{-1}(T^{i}x)}\right)\leq C,
		\end{equation}
		which is denoted by $\omega\in A^{exp}_{\infty}.$

		\item \label{Thm:con}
		There exist $0<\gamma,\delta<1$ such that for $a.e.$~$x$ and for all positive integers $k$ we have
		\begin{equation}\label{conn}
    	\frac{1}{k}\sharp\Big\{j:0\leq j \leq k-1;\omega(T^{j}x)\leq  \gamma\frac{1}{k}\sum\limits_{i=0}^{k-1} \omega(T^{i}x)\Big\}\leq\delta,
		\end{equation}
		which is denoted by $\omega\in A^{avg}_{\infty}.$
	\end{enumerate}
\end{theorem}

\begin{proof}[Proof of Theorem \ref{Thm:equa}] $\ref{Thm:equa_Ap}\Rightarrow\ref{Thm:equa_A_exp_infty}$ Let $\omega\in A_p$ with $p>1.$
Using Jensen's inequality, we have
$$\exp\left({\frac{1}{k}\sum\limits_{i=0}^{k-1}\log \omega^{-\frac{1}{p-1}}(T^{i}x)}\right)\leq\frac{1}{k}\sum\limits_{i=0}^{k-1} \omega^{-\frac{1}{p-1}}(T^{i}x).$$
Then $$\exp\left({\frac{\frac{1}{p-1}}{k}\sum\limits_{i=0}^{k-1}\log \omega^{-1}(T^{i}x)}\right)
\leq\frac{1}{k}\sum\limits_{i=0}^{k-1} \omega^{-\frac{1}{p-1}}(T^{i}x),$$
which implies
$$\exp\left({\frac{1}{k}\sum\limits_{i=0}^{k-1}\log \omega^{-1}(T^{i}x)}\right)\leq
		\left(\frac{1}{k}\sum\limits_{i=0}^{k-1} \omega^{-\frac{1}{p-1}}(T^{i}x)\right)^{p-1}.$$
It follows that
	\begin{eqnarray*}
&&\frac{1}{k}\sum\limits_{i=0}^{k-1} \omega(T^{i}x)\exp\left({\frac{1}{k}\sum\limits_{i=0}^{k-1}\log \omega^{-1}(T^{i}x)}\right)\\
&\leq&\frac{1}{k}\sum\limits_{i=0}^{k-1} \omega(T^{i}x)\left(\frac{1}{k}\sum\limits_{i=0}^{k-1} \omega^{-\frac{1}{p-1}}(T^{i}x)\right)^{p-1}\leq C.
	\end{eqnarray*}
Thus we have $$\frac{1}{k}\sum\limits_{i=0}^{k-1} \omega(T^{i}x)\exp\left({\frac{1}{k}\sum\limits_{i=0}^{k-1}\log \omega^{-1}(T^{i}x)}\right)\leq C.$$

$\ref{Thm:equa_A_exp_infty}\Rightarrow\ref{Thm:con}$ Let $k$ be fixed and let $v_k(x)=\exp\left(\frac{1}{k}\sum\limits_{i=0}^{k-1}\log \omega(T^{i}x)\right).$ We have that
\begin{eqnarray*}
	1&=&\frac{1}{v_k(x)}\exp\left(\frac{1}{k}\sum\limits_{i=0}^{k-1}\log \omega(T^{i}x)\right)\\
	&=&\exp\left(\frac{1}{k}\big(k\log\frac{1}{v_k(x)}\big)+\frac{1}{k}\sum\limits_{i=0}^{k-1}\log \omega(T^{i}x)\right)\\
	&=&\exp\left(\frac{1}{k}\sum\limits_{i=0}^{k-1}\log \frac{\omega(T^{i}x)}{v_k(x)}\right).
\end{eqnarray*}
It follows that
\begin{equation}\label{equal0}
	\frac{1}{k}\sum\limits_{i=0}^{k-1}\log \frac{\omega(T^{i}x)}{v_k(x)}=0.
\end{equation}
Using \eqref{A_exp_infty}, we obtain that
\begin{equation}\label{3.6}
	\frac{1}{v_k(x)}\frac{1}{k}\sum\limits_{i=0}^{k-1} \omega(T^{i}x)\leq  \frac{C}{v_k(x)} \exp\left(\frac{1}{k}\sum\limits_{i=0}^{k-1}\log \omega(T^{i}x)\right) =C.
\end{equation}
For some $\gamma>0$ to
be chosen later, we observe that
\begin{eqnarray*}
	&&\left \{j:0\leq j \leq k-1;\omega(T^{j}x)\leq  \gamma \frac{1}{k}\sum\limits_{i=0}^{k-1} \omega(T^{i}x)\right\} \\	
	&=&\left \{j:0\leq j \leq k-1;\frac{\omega(T^{j}x)}{v_k(x)}\leq  \frac{\gamma}{v_k(x)}\frac{1}{k}\sum\limits_{i=0}^{k-1} \omega(T^{i}x)\right\}\\		
	&\subseteq& \left\{j:0\leq j \leq k-1;\frac{\omega(T^{j}x)}{v_k(x)}\leq \gamma C\right\}\\	
	&=&\left\{j:0\leq j \leq k-1;\log\left(1+\frac{1}{\gamma C}\right)  \leq  \log\left(1+ \frac{v_k(x)}{\omega(T^{j}x)}\right) \right\}.
\end{eqnarray*}
Thus
\begin{eqnarray*}	
	&& \sharp\left\{j:0\leq j \leq k-1;\omega(T^{j}x)\leq \gamma \frac{1}{k}\sum\limits_{i=0}^{k-1} \omega(T^{i}x)\right\} \\	
	&\leq&\sharp\left\{j:0\leq j \leq k-1; \log(1+\frac{1}{\gamma C})  \leq  \log\left(1+\frac{v_k(x)}{\omega(T^{j}x)}\right) \right\}\\
	&\leq& \frac{1}{\log(1+\frac{1}{\gamma C})}\sum\limits_{i=0}^{k-1}\log\left(1+\frac{v_k(x)}{\omega(T^{i}x)}\right)\\	
	&=& \frac{k}{\log(1+\frac{1}{\gamma C})} \frac{1}{k}\sum\limits_{i=0}^{k-1}\log\left(\left(1+\frac{v_k(x)}{\omega(T^{i}x)}\right)\frac{\omega(T^{i}x)}{v_k(x)}\right)\\	
	&=& \frac{k}{\log(1+\frac{1}{\gamma C})} \frac{1}{k}\sum\limits_{i=0}^{k-1}\log \left(1+ \frac{\omega(T^{i}x)}{v_k(x)}\right),	
\end{eqnarray*}
where we have used \eqref{equal0}.
It follows from \eqref{3.6} that
\begin{eqnarray*}
	&& \frac{k}{\log(1+\frac{1}{\gamma C})} \frac{1}{k}\sum\limits_{i=0}^{k-1}\log \left(1+ \frac{\omega(T^{i}x)}{v_k(x)}\right)\\	       
	&\leq& \frac{k}{\log(1+\frac{1}{\gamma C})} \frac{1}{k}\sum\limits_{i=0}^{k-1}  \frac{\omega(T^{i}x)}{v_k(x)}\\		
	&\leq& \frac{kC}{\log(1+\frac{1}{\gamma C})}.
\end{eqnarray*}
Since
$\lim\limits_{\gamma\rightarrow0}\frac{C}{\log(1+\frac{1}{\gamma C})}=0,$ we have $\eqref{conn}$.

\end{proof}

\begin{lemma}\label{key_lemma_sta}Let $v$ be a positive measurable function. Then
\begin{equation}\label{key_lemma}
	\left(\frac{1}{k}\sum\limits_{i=0}^{k-1} v^{s}(T^{i}x)\right)^{\frac{1}{s}}\downarrow\exp \left(\frac{1}{k}\sum\limits_{i=0}^{k-1} \log v(T^{i}x)\right),~\text{~as~}s\downarrow0^+.
\end{equation}\end{lemma}

\begin{proof}[Proof of Lemma \ref{key_lemma_sta}] 
		H\"{o}lder's inequality gives
		$$\left(\frac{1}{k}\sum\limits_{i=0}^{k-1} v^s(T^{i}x)\right)^{\frac{1}{s}}\leq\left(\frac{1}{k}\sum\limits_{i=0}^{k-1}v^t(T^{i}x)\right)^{\frac{1}{t}}$$ with $0<s<t<s_0.$		
		By Jensen's inequality, we have
		$$
		\exp\frac{1}{k}\sum\limits_{i=0}^{k-1}\log v^s(T^{i}x)\leq\frac{1}{k}\sum\limits_{i=0}^{k-1} v^s(T^{i}x),
		$$
		which implies \begin{equation}\label{eq_left}\exp \left(\frac{1}{k}\sum\limits_{i=0}^{k-1} \log v(T^{i}x)\right)\leq\left(\frac{1}{k}\sum\limits_{i=0}^{k-1} v^s(T^{i}x)\right)^{\frac{1}{s}}.
		\end{equation}
		Because of $x\leq \exp(x-1)$ for $x>0,$ then
		$$\frac{1}{k}\sum\limits_{i=0}^{k-1} v^s(T^{i}x)\leq\exp \left(\frac{1}{k}\sum\limits_{i=0}^{k-1} v^s(T^{i}x)-1\right).$$
		It follows that
		\begin{eqnarray*}\label{eq_right}
			&&\left(\frac{1}{k}\sum\limits_{i=0}^{k-1} v^s(T^{i}x)\right)^{\frac{1}{s}}\\
                &\leq&\exp \Big(\frac{\frac{1}{k}\sum\limits_{i=0}^{k-1} v^s(T^{i}x)-1}{s}\Big)\\
			&=&\exp \frac{1}{k}\sum\limits_{i=0}^{k-1}\frac{v^s(T^{i}x)-1}{s}
                \downarrow\exp \left(\frac{1}{k}\sum\limits_{i=0}^{k-1} \log v(T^{i}x)\right),~~\text{~as~}s\downarrow0^+.
		\end{eqnarray*}
				This completes the proof of \eqref{key_lemma}
\end{proof}

Using a kind of reverse H\"{o}lder's condition which appeared in Str\"{o}mberg and Wheeden \cite{MR766221}, we give a characterization of $\omega\in A^{exp}_{\infty},$ which is Theorem \ref{thm:exp-s}.
\begin{theorem}\label{thm:exp-s}The following statements are equivalent.
	\begin{enumerate}[\rm (1)]
		\item \label{thm:exp-s2} $\omega\in A^{exp}_{\infty}.$
		\item \label{thm:exp-s1}There exists $C>1$ such that for every $s\in(0,1),$   $a.e.$~$x$ and for all positive integers $k$ we have
		\begin{equation}\label{thm:eq-exp-s}
			\frac{1}{k}\sum\limits_{i=0}^{k-1} \omega(T^{i}x)\leq C\left(\frac{1}{k}\sum\limits_{i=0}^{k-1} \omega^{s}(T^{i}x)\right)^{\frac{1}{s}},\end{equation}
		which is denoted by $\omega\in A^{SW}_{\infty}.$
	\end{enumerate}
\end{theorem}

\begin{proof}[Proof of Theorem \ref{thm:exp-s}]
	In view of Lemma \ref{key_lemma_sta}, we have that
	\begin{equation*}
	\left(\frac{1}{k}\sum\limits_{i=0}^{k-1} \omega^{s}(T^{i}x)\right)^{\frac{1}{s}}\downarrow\exp \left(\frac{1}{k}\sum\limits_{i=0}^{k-1} \log\omega(T^{i}x)\right),~\text{~as~}s\downarrow0^+,
	\end{equation*}
	which establishes the
	equivalence between \ref{thm:exp-s2} and \ref{thm:exp-s1}.	
\end{proof}

For $A_{\infty}^{med}$ in \cite{MR529683}, we replace the median $m(\omega;Q)$ by the median function $m(\omega,k)$ (see Definition \ref{media_f}), which is the key observation in Theorem \ref{thm:imp2}.

\begin{definition}\label{media_f}The median function of $\omega$ relative to $k$ is defined as a measurable function $m(\omega;k)$ such that $\sharp\{i:0\leq i\leq k-1;~\omega(T^{i}x)< m(\omega, k)\}\leq k/ 2$
and $\sharp\{i:0\leq i\leq k-1;~\omega(T^{i}x)> m(\omega, k)\}\leq  k/ 2.$

\end{definition}

\begin{theorem}\label{thm:imp2}Let $\omega$ be a weight. We have the sequence of implications
	$\eqref{thm:imp2_Log} \Rightarrow\eqref{thm:imp2_Mid}\Rightarrow\eqref{thm:imp2_Dou}$
	for the following statements.
	\begin{enumerate}
	
		\item \label{thm:imp2_Log}$\omega\in A^{exp}_{\infty}.$
		
		\item \label{thm:imp2_Mid} There exists $C>1$ such that for all $n\in \mathbb{N}$ we have
		\begin{equation}\label{thm:imp2_EMid}
			T_{0, k-1} \omega(x)\leq Cm(\omega,k),
		\end{equation}
		which is denoted by $\omega\in A_{\infty}^{med}.$
		
		\item \label{thm:imp2_Dou}$\omega \in A_{\infty}^{M}.$
	\end{enumerate}
\end{theorem}

\begin{proof}[Proof of Theorem \ref{thm:imp2}] 
	
	$\eqref{thm:imp2_Log} \Rightarrow\eqref{thm:imp2_Mid}$
	Let $E=\{j:0\leq j\leq k-1; \omega(T^{j}x)>m(\omega,k)\}.$ Using H\"{o}lder's inequality,
	we have
	\begin{eqnarray*}
		\frac{1}{k}\sum\limits_{j\in E} \omega ^s(T^{j}x) 
		&\leq&\left(\frac{1}{k}\sum\limits_{i=0}^{k-1} \omega (T^{i}x)\right)^s(\frac{\sharp E}{k})^{1-s}\\
		&\leq&2^{s-1}C^s\frac{1}{k}\sum\limits_{i=0}^{k-1} \omega ^s(T^{i}x) ,
	\end{eqnarray*}
	where we have used Theorem \ref{thm:exp-s}.
	It follows that $$\frac{1}{k}\sum\limits_{j\in E} \omega ^s(T^{j}x)
	\leq\frac{3}{4}\frac{1}{k}\sum\limits_{i=0}^{k-1} \omega ^s(T^{i}x) $$ provided $2^{s-1}C^s<\frac{3}{4}.$
	Then $\frac{1}{k}\sum\limits_{j\in E^c} \omega ^s(T^{j}x) 
	\geq\frac{1}{4}\frac{1}{k}\sum\limits_{i=0}^{k-1} \omega ^s(T^{i}x) .$ Thus
	\begin{eqnarray*}
		\frac{1}{4}\left(\frac{1}{k}\sum\limits_{i=0}^{k-1} \omega (T^{i}x)\right)^s &\leq&\frac{1}{4}C^s\frac{1}{k}\sum\limits_{i=0}^{k-1} \omega ^s(T^{i}x) \\ &\leq&C^s\frac{1}{k}\sum\limits_{j\in E^c} \omega ^s(T^{j}x)\\
		&\leq&C^s\frac{1}{k}\sum\limits_{j\in E^c} m(\omega, k) ^s\\
		&\leq&C^s\frac{1}{k}\sum\limits_{i=0}^{k-1} m(\omega, k) ^s \\
		&=&C^s m(\omega, k) ^s,
	\end{eqnarray*}
	which implies $\frac{1}{k}\sum\limits_{i=0}^{k-1} \omega(T^{i}x) \leq4^{\frac{1}{s}}C m(\omega, k).$
	
	$\eqref{thm:imp2_Mid}\Rightarrow\eqref{thm:imp2_Dou}$ 
Let $\alpha<\frac{1}{4}$ and $\frac{\sharp E}{k}<\frac{1}{4}.$ We claim that $$\frac{\sharp(E^c\cap\{i:0\leq i\leq k-1;\omega(T^{i}x)\geq  m(\omega, k)\} )}{k} \geq\frac{1}{4}.$$
	Indeed, we have
	\begin{eqnarray*}&&\frac{\sharp ({E\cap\{i:0\leq i\leq k-1;\omega(T^{i}x)< m(\omega, k)\}}}{k} \\
		&\leq&\frac{\sharp E}{k}+\frac{\sharp (\{i:0\leq i\leq k-1;\omega(T^{i}x)< m(\omega, k)\})}{k} \\
		&<&\frac{1}{4}+\frac{1}{2}=\frac{3}{4}.\end{eqnarray*}
	This proves that $$\frac{\sharp(E^c\cap\{i:0\leq i\leq k-1;\omega(T^{i}x)\geq  m(\omega, k)\} )}{k} \geq\frac{1}{4}.$$
	It follows from $\eqref{thm:imp2_Mid}$ that
	\begin{eqnarray*}
		\frac{1}{k}\sum\limits_{i=0}^{k-1} \omega(T^{i}x)
		&\leq& Cm(\omega, k)\\
		&\leq& 4C m(\omega, k)\frac{\sharp(E^c\cap\{i:0\leq i\leq k-1;\omega(T^{i}x)\geq  m(\omega, k)\} )}{k} \\
		&\leq& 4C \frac{1}{k}\sum\limits_{j\in E^c}\omega(T^{j}x).
	\end{eqnarray*}
Then we have $\sum\limits_{i=0}^{k-1} \omega(T^{i}x)\leq 4C \sum\limits_{j\in E^c}\omega(T^{j}x) $ which implies $$\frac{1}{4C}\leq    	\frac{	\sum\limits_{j \in E^c} \omega(T^{j} x)}{\sum\limits_{i=0}^{k-1} \omega(T^{i} x)}.$$ Thus $\frac{\sum\limits_{i \in E} \omega(T^{i} x)}{\sum\limits_{i=0}^{k-1} \omega(T^{i} x)}  <\beta$ with $\beta=1-\frac{1}{4C}.$
\end{proof}

As we did in Section \ref{Sec-RH}, we can develop and prove theorems involving $g.$ Because these results will not be used in the rest of our paper, we omit them.
At the end of this section, our results are collected in Figure \ref{figure_Ap}. 
Combining with Section \ref{Sec-RH}, we obtain that $\cup_{p>1}A_p\Rightarrow \cup_{q>1}RH_{q}.$ The converse will be studied in Section \ref{Sec_fur}.
\begin{center}
\begin{figure}[H]
	\begin{tikzpicture}[scale=0.75]	
     \draw (1,5) node {\boxed{\cup_{p>1}A_p}} -- (1,5);
	\draw (4,3) node {\boxed{A^{exp}_{\infty}}} -- (4,3);
	\draw (4,1) node {\boxed{A^{SW}_{\infty}}} -- (4,1);
	\draw (7,1) node  {\boxed{ A^{med}_{\infty}}}-- (7,1);
	\draw (7,3) node  {\boxed{ A^{avg}_{\infty}}}-- (7,3);
	\draw (10,1) node  {\boxed{ A^{M}_{\infty}}}-- (10,1);
    \draw[-implies,double equal sign distance, line width=0.25mm] (2.1,4.5)--(3.3,3.5) ;  
    \draw[-implies,double equal sign distance, line width=0.25mm] (4.7,3)--(6.2,3) ;          
    \draw[implies-implies,double equal sign distance, line width=0.25mm] (4,1.5)--(4,2.5) ;  
    \draw[-implies,double equal sign distance, line width=0.25mm] (4.7,2.5)--(6.2,1.5) ; 
    \draw[-implies,double equal sign distance, line width=0.25mm] (7.8,1)--(9.3,1) ;    
	\end{tikzpicture}
\caption{Properties of the Union of $A_p$ Weights}
	\label{figure_Ap}
	\end{figure}
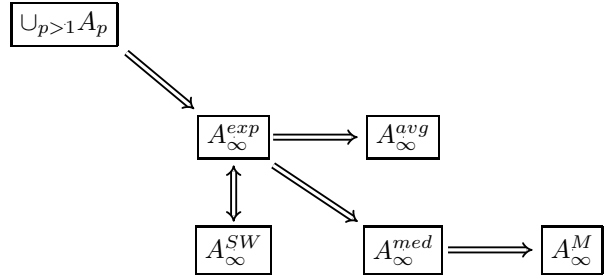
\end{center}

\section{Further results with doubling condition}\label{Sec_fur}
In this section we show $\cup_{q>1}RH_{q}\Rightarrow \cup_{p>1}A_p$ under the doubling condition on the weight $\omega.$ This is theorem \ref{Thm_RH_AP}. It seems reasonable to assume that $\omega$ satisfies the doubling condition, because $\omega\in A_p$ implies the doubling condition, which will be shown in the following Lemma \ref{lem-dou}.

\begin{lemma}\label{lem-dou} 
let $\omega$ be a weight and $p>1$. If $\omega\in A_p,$ then
$\omega$ satisfies the doubling condition.
\end{lemma}

\begin{proof}[Proof of Lemma \ref{lem-dou}] Let $I_i$ be children of $\{0,1,\ldots,k-1\}$ with $k\geq2.$ By H\"{o}lder's inequality and $\omega\in A_p,$ we have
\begin{eqnarray*}	
\frac{\sharp I_i}{\sharp I}&\leq&\frac{\left(\sum_{j \in I_{i}}\omega(T^{j} x)\right)^{\frac{1}{p}}\left(\sum_{j \in I_{i}}\omega^{-\frac{1}{p-1}}(T^{j}x)\right)^{\frac{1}{p\prime}}}{\sharp I}\\
&\leq&\left(\frac{\sum_{j \in I_{i}}\omega(T^{j} x)}{\sum_{j \in I}\omega(T^{j} x)}
\frac{\left(\sum_{j \in I}\omega(T^{j} x)\right)\left(\sum_{j \in I_{i}}\omega^{-\frac{1}{p-1}}(T^{j}x)\right)^{\frac{p}{p\prime}}}{(\sharp I)^p}\right)^{\frac{1}{p}}\\
&\leq&C^{\frac{1}{p}}\left(\frac{\sum_{j \in I_{i}}\omega(T^{j} x)}{\sum_{j \in I}\omega(T^{j} x)}
\right)^{\frac{1}{p}}.
\end{eqnarray*}	
It follow from $\frac{\sharp{I_i}}{\sharp{I}}\geq\frac{1}{3}$ that 
$$\sum_{j \in I}\omega(T^{j} x)\leq 3^{\frac{1}{p}}C\sum_{j \in I_{i}}\omega(T^{j} x).$$
\end{proof}

We now turn to Theorem $\ref{Thm_RH_AP}.$ 

\begin{theorem}\label{Thm_RH_AP} Let $\omega$ be a weight. If $\omega$ satisfies the doubling condition, then
$\cup_{q>1}RH_{q}\Rightarrow \cup_{p>1}A_p.$
\end{theorem}

\begin{proof}[Proof of Theorem \ref{Thm_RH_AP}]
It follows from the definitions of $\hat{A}^{M}_{\infty}$ and $A^{M}_{\infty}(\omega)$ that
\begin{equation}
\label{eq_A_infty}\omega\in \hat{A}^{M}_{\infty}\Leftrightarrow\omega^{-1}\in A^{M}_{\infty}(\omega).
\end{equation}
Letting $\omega\in\cup_{q>1}RH_{q},$ we have
\begin{center}
\begin{figure}[H]
	\begin{tikzpicture}[scale=1]
	\draw (0,0) node {$\omega\in\cup_{q>1}RH_{q}$} -- (0,0);
     \draw (3,0) node {$\omega\in \hat{A}^{M}_{\infty}$} -- (3,0);
     \draw (6,0) node {$\omega^{-1}\in A^{M}_{\infty}(\omega)$} -- (6,0); 
     \draw (9.9,0) node {$\omega^{-1}\in \cup_{q>1}RH_{q}(\omega).$} -- (9.9,0); 
	\draw (1.7,0.5) node {Thm \ref{sThm_equa_RH}} -- (1.7,0.5);
     \draw (7.7,0.5) node {Thm \ref{Thm_equa_RH}} -- (7.7,0.5);
     \draw (4.3,0.5) node {\eqref{eq_A_infty}} -- (4.3,0.5);

    \draw[implies-implies,double equal sign distance, line width=0.25mm] (1.1,0) -- (2.3,0);
    \draw[implies-implies,double equal sign distance, line width=0.25mm] (3.7,0) -- (4.9,0);
    \draw[implies-implies,double equal sign distance, line width=0.25mm] (7.1,0) -- (8.3,0);
	\end{tikzpicture}
	\end{figure} 
\end{center}
Then $\omega^{-1}\in RH_{q}(\omega)$ gives that
\begin{eqnarray*}	
			&~&\left(\frac{1}{\sum_{i=0}^{k-1}\omega(T^{i} x)} \sum\limits_{i=0}^{k-1}\big(\omega^{-1}(T^{i} x)\big)^q\omega(T^{i} x)\right )^{\frac{1}{q}}\\
			&\leqslant& C \left(\frac{1}{\sum_{i=0}^{k-1}\omega(T^{i} x)} \sum\limits_{i=0}^{k-1}\big(\omega^{-1}(T^{i} x)\big)\omega(T^{i} x)\right ).
\end{eqnarray*}
Hence $$ \frac{1}{k}\sum_{i=0}^{k-1} \omega(T^{i} x)\left(\frac{1}{k} \sum_{i=0}^{k-1}
\omega^{-\frac{1}{p-1}}(T^{i} x)\right)^{p-1} \leq C^p,$$
where $\frac{1}{p}+\frac{1}{q}=1.$ This is $\omega\in A_p.$ 
Under the doubling assumption on the weight $\omega,$ we have shown that $\cup_{q>1}RH_{q}\Rightarrow \cup_{p>1}A_p.$
\end{proof}

Thus we conclude that the union of $A_p$ weights
has the following characterizations in Figure \ref{figure_all}.
\begin{center}
\begin{figure}[H]
	\begin{tikzpicture}[scale=0.9,auto]	
     \draw (0,5) node {\boxed{\cup_{p>1}A_p}} -- (0,5);
	\draw (2,3) node {\boxed{A^{exp}_{\infty}}} -- (2,3);
	\draw (2,1) node {\boxed{A^{SW}_{\infty}}} -- (2,1);
	\draw (4,1) node  {\boxed{ A^{med}_{\infty}}}-- (4,1);

\draw (4,3) node {\boxed{A^{avg}_{\infty}}} -- (4,3);
	\draw (6,5) node {\boxed{A^{\lambda}_{\infty}}} -- (6,5);
	\draw (9.5,5) node {\boxed{\cup_{q>1}RH_{q}}} -- (9.5,5);
    \draw (12,3) node  {\boxed{A^{CF}_{\infty}}}-- (12,3);
    \draw (10,1) node  {\boxed{ \hat{A}^{M}_{\infty}}}-- (10,1);
	\draw (6,1) node  {\boxed{ A^{M}_{\infty}}}-- (6,1);
	\draw (8,3) node {\boxed{A^{log}_{\infty}}} -- (8,3);

\draw[-implies, double equal sign distance, red, line width=0.25mm] (8.4,5.1)to [bend right=35] node [above] {Doubling  Condition}(0.9,5.1);

    \draw[-implies,double equal sign distance, line width=0.25mm] (0.9,4.6)--(1.8,3.5) ;  
   \draw[-implies,double equal sign distance, line width=0.25mm] (2.6,3)--(3.4,3) ;          
    \draw[implies-implies,double equal sign distance, line width=0.25mm] (2,1.5)--(2,2.5) ;  
    \draw[-implies,double equal sign distance, line width=0.25mm] (2.6,2.6)--(3.6,1.5) ; 
    \draw[-implies,double equal sign distance, line width=0.25mm] (4.7,1)--(5.5,1) ;    

\draw[blue, -implies,double equal sign distance, line width=0.25mm] (4,3.5) -- (5.5,4.9);
     \draw[blue, -implies,double equal sign distance, line width=0.25mm] (6.5,5) -- (8.4,5);
     \draw[blue, -implies,double equal sign distance, line width=0.25mm] (10.6,4.9) -- (12,3.5);
    \draw[blue, -implies,double equal sign distance, line width=0.25mm] (12,2.5) -- (10.6,1);
   \draw[blue, -implies,double equal sign distance, line width=0.25mm] (9.5,1) -- (6.5,1);
   \draw[blue, -implies,double equal sign distance, line width=0.25mm] (5.5,1.2) -- (4,2.5);    
   \draw[blue, -implies,double equal sign distance, line width=0.25mm] (9.5,4.5) -- (8.6,3.5);
   \draw[blue, -implies,double equal sign distance, line width=0.25mm] (7.6,2.5) -- (6.55,1.45);
	\end{tikzpicture}
\caption{Characterizations of $A_\infty$ Weights with the Doubling Condition}
	\label{figure_all}
	\end{figure}
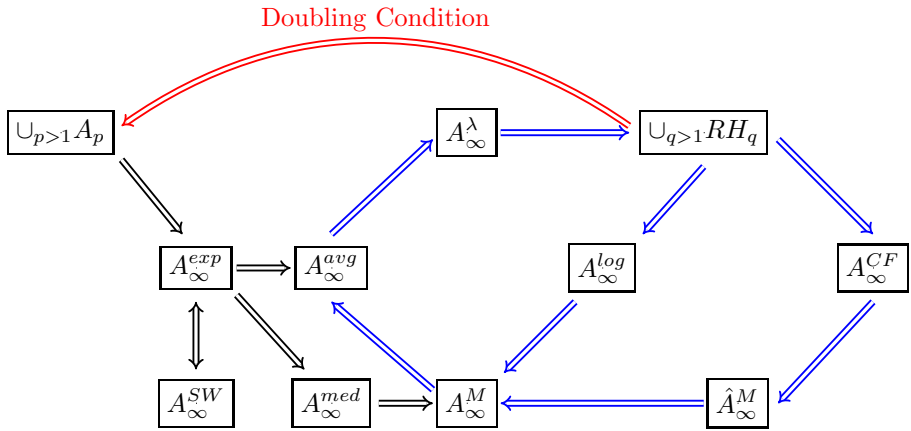
\end{center}

\bibliographystyle{alpha,amsplain}	

\end{document}